\documentclass [11pt,twoside]{amsart}
\usepackage{setspace}
\usepackage {amsfonts}
\usepackage {amsmath}
\usepackage {amsthm}
\usepackage {amssymb}
\usepackage {framed}
\usepackage {amsxtra}
\usepackage {enumerate}
\usepackage {graphicx}
\usepackage{color}
\usepackage{graphicx}
\usepackage[left=1.6in, right=1.6 in, top=1.0 in, footskip=0.5 in, bottom=1.0 in]{geometry}
\makeatletter

\theoremstyle{definition}
\newtheorem{df}{Definition} [section]

\theoremstyle{plain}
\newtheorem{thm}[df]{Theorem}

\newtheorem{lemma}[df]{Lemma}
\newtheorem{cor}[df]{Corollary}

\newtheorem{problem}[df]{Problem}

\title{Grid Dissections of Tangential Quadrilaterals}

\author{Erica Choi}
\address{Blair Academy, Blairstown, NJ 07825}
\email{choie@blair.edu}

\author{Dan Ismailescu}
\address{Mathematics Department, Hofstra University, Hempstead, NY 11549}
\email{dan.p.ismailescu@hofstra.edu}

\author{Jiho Lee}
\address{Canterbury School, New Milford, CT 06776}
\email{jhlee00502@gmail.com}

\author{Joonsoo Lee}
\address{Dwight Englewood School, Englewood, NJ 07631}
\email{jlee20@d-e.org}

\begin{document}

\begin{abstract}
For any integer $n\ge 2$, a square can be partitioned into $n^2$ smaller squares
via a {\it checkerboard-type} dissection.
Does there such a \emph{class preserving grid dissection} exist for some other types of quadrilaterals?
For instance, is it true that a tangential quadrilateral can be partitioned into $n^2$ smaller tangential quadrilaterals using 
an $n\times n$ grid dissection? We prove that the answer is affirmative for every integer $n\ge 2$.
\end{abstract}

\maketitle

\begin{section}{\bf The problem}

For any integer $n\ge 2$, a square can be naturally partitioned into $n^2$ smaller squares
via a {\it checkerboard-type} dissection. Are there any other properties $\mathcal{P}$ such that any convex quadrilateral
$Q$ having property $\mathcal{P}$ can be dissected in a {\it grid-like} manner into smaller quadrilaterals, all of which have property $\mathcal{P}$ as well?

The answer is obvious if property $\mathcal{P}$ says ``{\it $Q$ is a parallelogram}" or ``{\it $Q$ is a trapezoid}"
or ``{\it $Q$ is a convex quadrilateral}".

What if property $\mathcal{P}$ states: ``{\it $Q$ is a \emph{cyclic} quadrilateral}" or
''{\it $Q$ is a \emph{tangential} quadrilateral}" or ``{\it $Q$ is an \emph{orthodiagonal} quadrilateral}"?

For instance, figure \ref{fig1} displays the case of a tangential quadrilateral partitioned into nine smaller
tangential quadrilaterals. Is such a dissection always possible? And if it is, how do we find it?

\begin{figure}[!htb]
\centering
\includegraphics[scale=1]{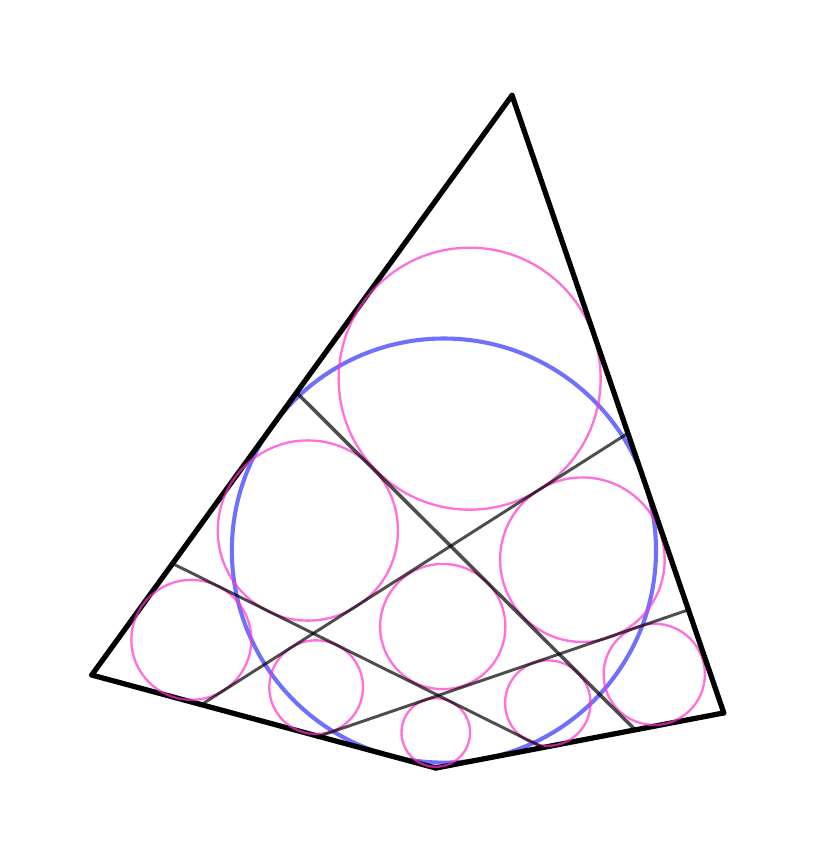}
\vspace{-1cm}
\caption{\small A $3\times 3$ grid dissection of a tangential quadrilateral into smaller tangential quadrilaterals}
\label{fig1}
\end{figure}
\end{section}

\begin{section}{\bf Class preserving grid dissections of quadrilaterals}\label{defback}

\begin{df}
Let $ABCD$ be a convex quadrilateral and let $m$, $n$ be two positive
integers. Consider two sets of segments $\mathcal{S}=\{s_1,\,s_2,\ldots,\,s_{m-1}\}$
and $\mathcal{T}=\{t_1,\,t_2,\ldots,\,t_{n-1}\}$ with the following properties:

a) If $s\in \mathcal{S}$ then the endpoints of $s$ belong to the sides $AB$ and $CD$.
Similarly, if $t\in \mathcal{T}$ then the endpoints of $t$ belong to the sides $AD$ and $BC$.

b) Every two segments in $\mathcal{S}$ are pairwise disjoint and the same is true for the segments in $\mathcal{T}$.

\noindent We then say that segments  $s_1,\,s_2,\ldots,s_{m-1},\,t_1,\,t_2,\,\ldots, t_{n-1}$ define an $m\times n$
{\it grid dissection} of $ABCD$ -- see figure \ref{fig4}. Note that either $\mathcal{S}$ or $\mathcal{T}$ could be empty.
\end{df}
\begin{figure}[!htb]
\centering
\includegraphics[scale=0.7]{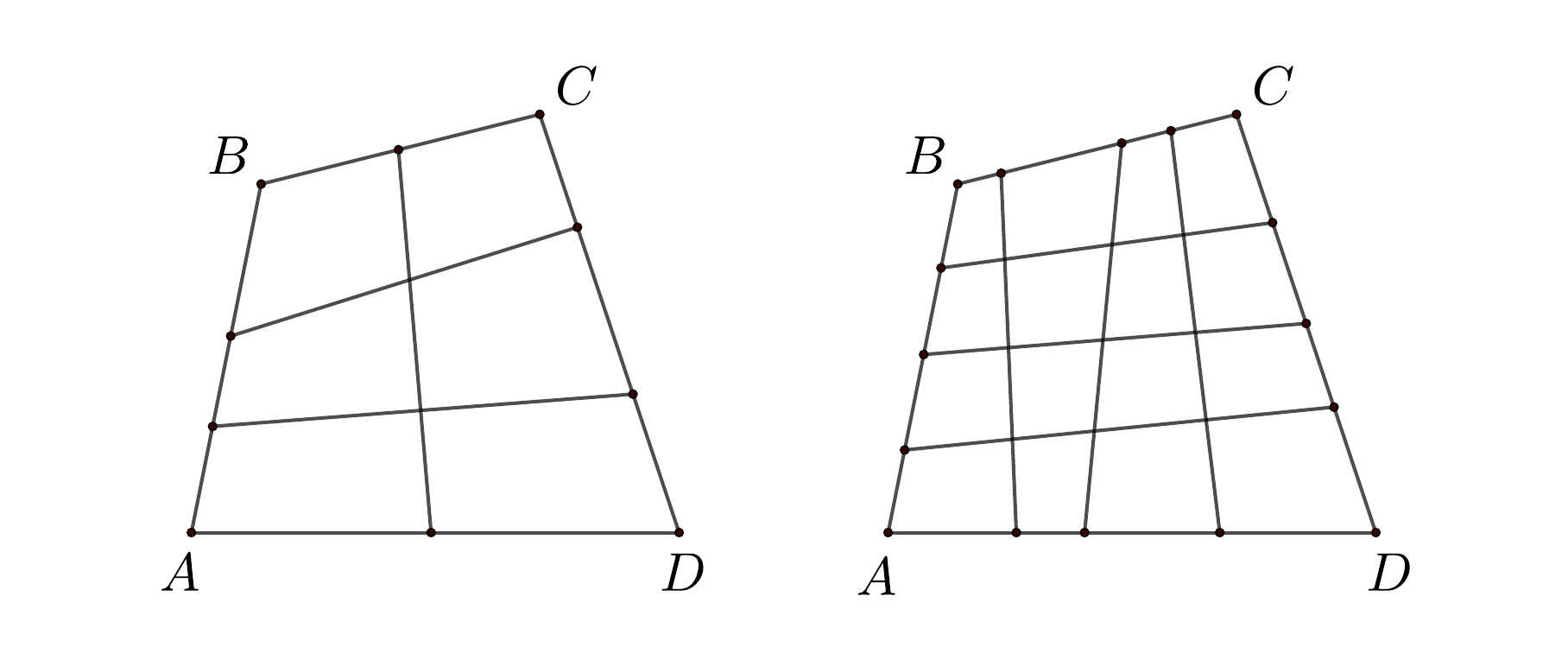}
\caption{\small A $3\times 2$ and a $4\times 4$ grid dissection of the convex quadrilateral $ABCD$}
\label{fig4}
\end{figure}

The following problem raised in \cite {IV} is the main motivation of our paper.
\vspace{-.2cm}
\begin{problem}\label{IVproblem}
\noindent Is it true that every cyclic, orthodiagonal or
tangential quadrilateral can be partitioned into cyclic,
orthodiagonal, or tangential quadrilaterals, respectively,
via an $m\times n$ grid dissection?
\end{problem}

The authors call such dissections \emph{class preserving grid dissections}.

In \cite {IV} it is proved that if either $m$ or $n$ is even then an $m\times n$ grid dissection of a cyclic quadrilateral $Q$
into $mn$ cyclic quadrilaterals exists only if $Q$ is either an isosceles trapezoid or a rectangle. The situation is a little bit better
if both $m$ and $n$ are odd.
\begin{thm}(\cite{IV})

(a) Every cyclic quadrilateral all of whose angles are
greater than $\arccos((\sqrt{5}-1)/{2})\approx 51.83^{\circ}$ admits
a $3\times 1$ grid dissection into three cyclic quadrilaterals.

(b) Let $Q$ be a cyclic quadrilateral such that the measure
of each of the arcs determined by the vertices of $Q$ on the circumcircle is greater
than $60^{\circ}$. Then $ABCD$ admits a $3\times 3$ grid dissection into nine cyclic
quadrilaterals.
\end{thm}

It is conjectured however, that when $m$ and $n$ both odd become large, a cyclic quadrilateral admits
an $m\times n$ grid dissection into cyclic quadrilaterals only if it is ``close'' to being a rectangle.

Things get even worse for class preserving grid dissections of orthodiagonal quadrilaterals.
It is conjectured that with the possible exception of some special cases, such dissections do not exist.

After these mostly negative results, it comes as a surprise that there is some hope when dealing with tangential quadrilaterals.

Ismailescu and Vojdany \cite{IV} proved that every tangential quadrilateral has a $2\times 2$ class preserving grid dissection and asked
whether this result can be extended for $n\times n$ grid dissections.

Surprisingly enough, the answer is affirmative and this constitutes the main result of our paper.
The following elementary result provides a characterization of tangential quadrilaterals and will be used in the sequel.
The direct statement is due to Pitot and the converse to Steiner.

\begin{thm}(see e. g. \cite {yiu})\label{Pitot}
A quadrilateral is tangential if and only if the sum of two opposite sides equals
the sum of the other two opposite sides.
\end{thm}
\end{section}

\begin{section}{\bf The main result}
\begin{thm} For any integer $n\ge 2$ and for any tangential quadrilateral $Q$, there exists an $n\times n$ grid dissection of $Q$ into $n^2$ smaller tangential quadrilaterals.
\end{thm}

The idea of the proof is natural. Given a checkerboard dissected axes-parallel square $ABCD$ in the $xy$ plane and a ``target'' tangential quadrilateral $A'B'C'D'$ in the $uv$ plane (see figure \ref {fig5} below), we are trying to find a geometric transformation $T$ with the following properties:

\begin{figure}[!htb]
\centering
\includegraphics[scale=.54]{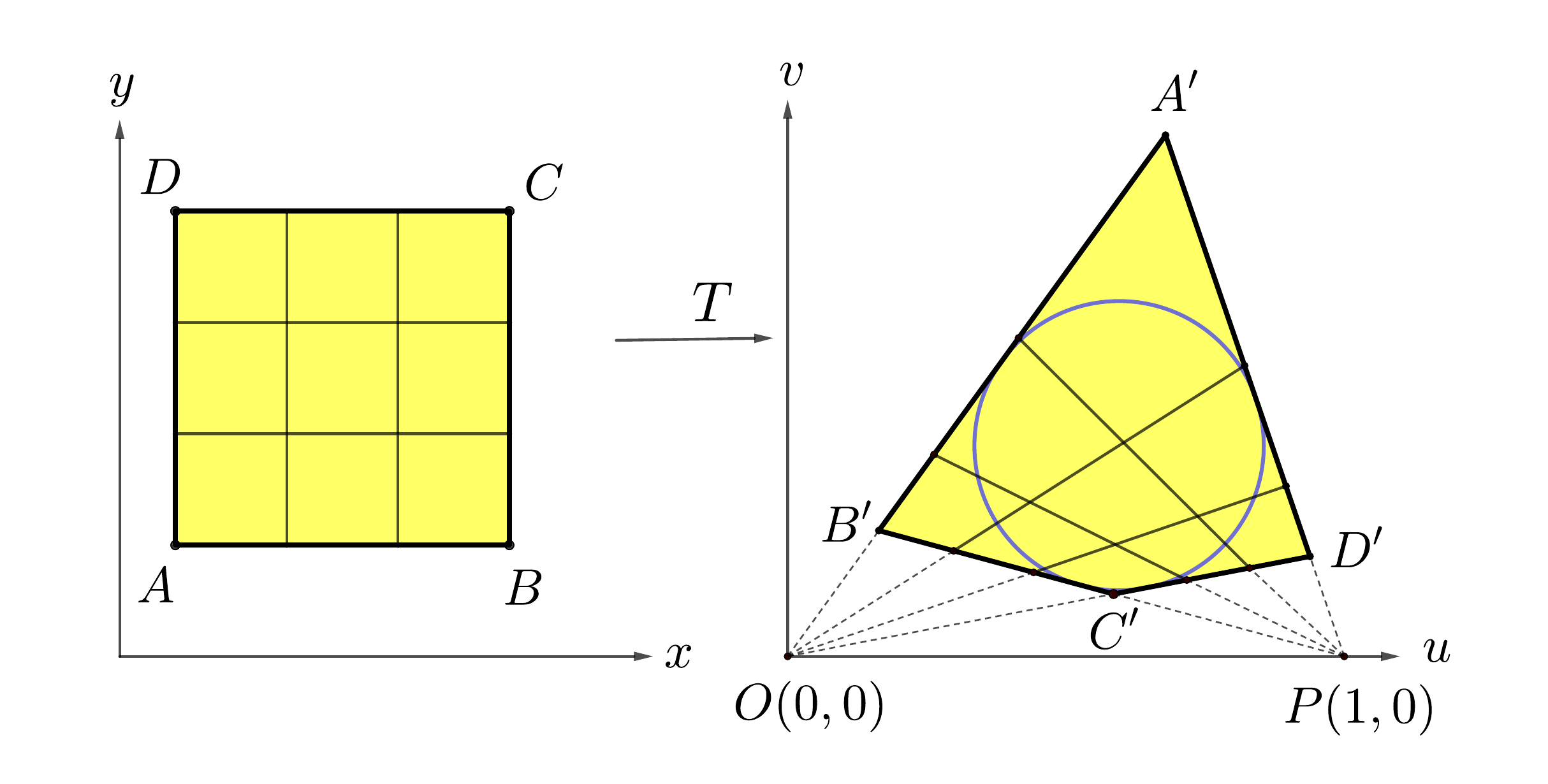}
\caption{\small The main idea behind transformation $T$}
\label{fig5}
\end{figure}

\begin{itemize}
\item{} $T$ maps arbitrarily small axes-parallel squares from the $xy$ plane into tangential quadrilaterals in the $uv$ plane.
\item{} $T$ maps the horizontal grid lines from the $xy$ plane into the grid lines passing through the point $O(0,0)$ in the $uv$ plane.
\item{} $T$ maps the vertical grid lines from the $xy$ plane into the grid lines passing through $P(1,0)$ in the $uv$ plane.

\end{itemize}

Of course, at this point we do not know whether such a transformation exists, and even if it does, whether it achieves the overall goal of generating
the desired $n\times n$ grid dissection of $A'B'C'D'$ into smaller tangential quadrilaterals.
\end{section}

\begin{section}{\bf The local condition}

\begin{lemma} Suppose that there exists a transformation $T:\mathbb{R}^2\rightarrow \mathbb{R}^2$ from the $xy$
plane to the $uv$ plane defined as $T:\,\,u=f(x,y)$, $v=g(x,y)$ such that $T$
maps every square from the $xy$ plane into a tangential quadrilateral in the $uv$ plane. Then
\begin{equation}\label{fxgx}
f^2_x(a,\,b)+g^2_x(a,\,b)=f^2_y(a,\,b)+g^2_y(a,\,b)\quad {\rm for\,\, every}\,\, (a,\,b) \in \mathbb{R}^2.
\end{equation}
\end{lemma}
\begin{proof}
\begin{figure}[!htb]
\centering
\includegraphics[scale=.5]{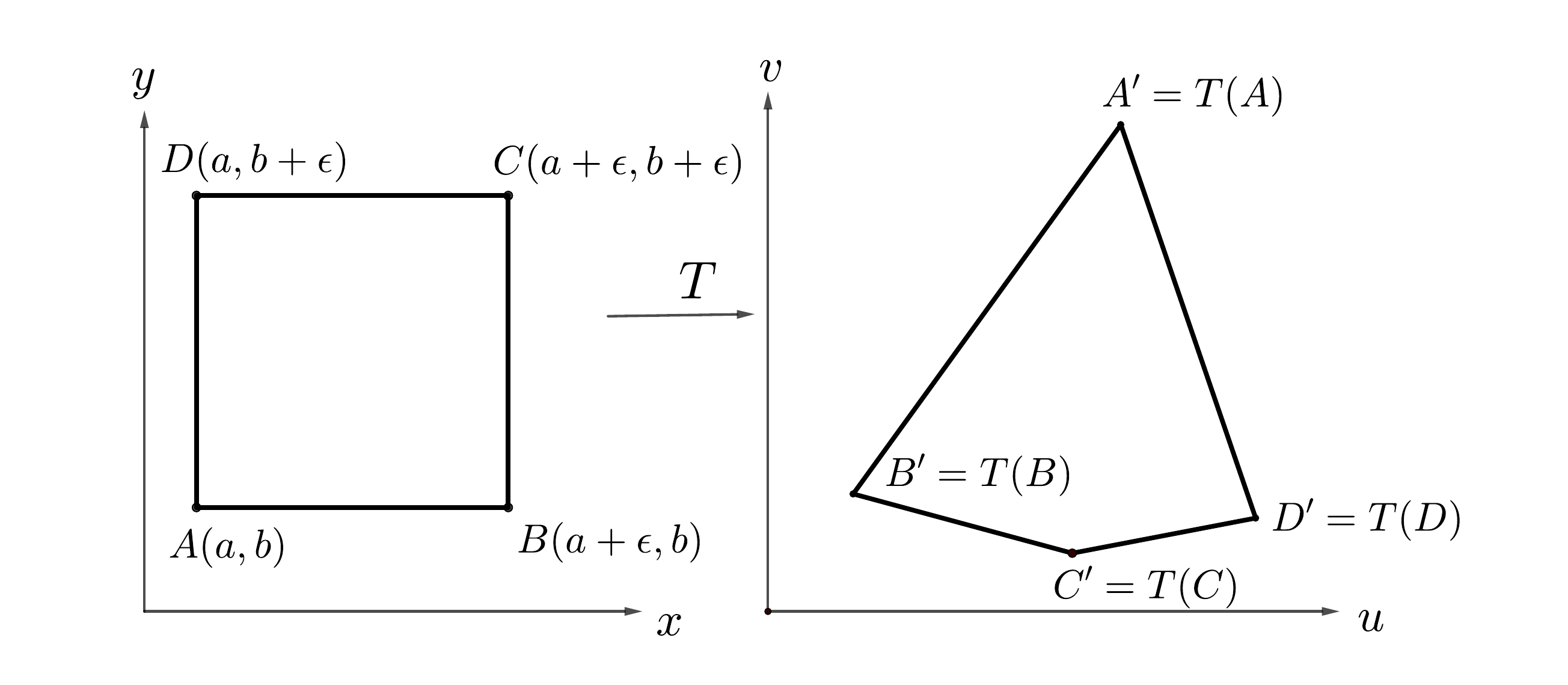}
\vspace{-0.5cm}
\caption{\small The local condition}
\label{fig6}
\end{figure}
Let $ABCD$ be an axis-parallel square of side $\epsilon$ in the $xy$ plane and let $A'=T(A)$, $B'=T(B)$, $C'=T(C)$
and $D'=T(D)$ as shown in figure \ref{fig6}. It follows that $u_{A'}=f(a,\,b)$, $v_{A'}=g(a,\,b)$,
$u_{B'}=f(a+\epsilon,\,b)$, $v_{B'}=g(a+\epsilon,\,B)$ etc. Consider first the ratio
\begin{eqnarray}
\frac{A'B'^2}{\epsilon^2}=\left[\frac{f(a+\epsilon,b)-f(a,b)}{\epsilon}\right]^2&+&\left[\frac{g(a+\epsilon,b)-
g(a,b)}{\epsilon}\right]^2\quad {\rm from\,\, which}\nonumber\\
\lim_{\epsilon\rightarrow 0} \left(\frac{A'B'}{\epsilon}\right) &=& \sqrt{f^2_x(a,b)+g^2_x(a,b)}\label{MN}.
\end{eqnarray}
Similarly,
\begin{eqnarray}
\frac{C'D'^2}{\epsilon^2}=\left[\frac{f(a+\epsilon,b+\epsilon)-f(a,b+\epsilon)}{\epsilon}\right]^2&+&
\left[\frac{g(a+\epsilon,b+\epsilon)-g(a,b+\epsilon)}{\epsilon}\right]^2\rightarrow\nonumber\\
\lim_{\epsilon\rightarrow 0} \left(\frac{C'D'}{\epsilon}\right) &=& \sqrt{f^2_x(a,b)+g^2_x(a,b)}\label{PQ}.
\end{eqnarray}

Analogously, we obatain the following equalities:
\begin{eqnarray}
\frac{B'C'^2}{\epsilon^2}&=&\left[\frac{f(a+\epsilon,b+\epsilon)-f(a+\epsilon,b)}{\epsilon}\right]^2+
\left[\frac{g(a+\epsilon,b+\epsilon)-g(a+\epsilon,b)}{\epsilon}\right]^2,\nonumber\\
\frac{D'A'^2}{\epsilon^2}&=&\left[\frac{f(a,b+\epsilon)-f(a,b)}{\epsilon}\right]^2+\left[\frac{g(a,b+\epsilon)
-g(a,b)}{\epsilon}\right]^2, \,\,\text{which give}\nonumber\\
&&\lim_{\epsilon\rightarrow 0} \left(\frac{B'C'}{\epsilon}\right) = \lim_{\epsilon\rightarrow 0}
\left(\frac{D'A'}{\epsilon}\right) = \sqrt{f^2_y(a,b)+g^2_y(a,b)}\label{NPQM}.
\end{eqnarray}

Since we want $A'B'C'D'$ to be tangential, by using the Pitot-Steiner condition it is necessary and sufficient to have that $A'B'+C'D'=B'C'+D'A'$. Divide both sides by $\epsilon$
and pass to the limit as  $\epsilon\rightarrow 0$ then use relations \eqref {MN}, \eqref {PQ} and \eqref {NPQM} to obtain
the desired equality.
 \end{proof}
\end{section}


\begin{section}{\bf The trapezoid case.}
\begin{thm}\label{trapezoidthm}
There exists a transformation, $T:\mathbb{R}^2\rightarrow \mathbb{R}^2$ with the following properties

\noindent(a) $T$ maps vertical lines from the $xy$ plane into vertical lines in the $uv$ plane.\\
(b) $T$ maps horizontal lines from the $xy$ plane into lines through the origin in the $uv$ plane.\\
(c) $T$ maps any axes-parallel square from the $xy$ plane into a tangential trapezoid in the $uv$ plane.\\
(d) For any given tangential trapezoid $K$ in the $uv$ plane, an appropriately scaled and/or rotated copy of $K$
can be obtained as the image under transformation $T$ of a conveniently chosen square in the $xy$ plane.
\end{thm}
\begin{proof}
Let $T:u=f(x,y),v=g(x,y)$
Any vertical line $x=c$ in the $xy$ plane is mapped into a vertical line $u=c'$, in the $uv$ plane.
This means $f(c,y)=c'$, where $c'$ depends on $c$ only. It follows that $f(x,y)$ depends on $x$ only, that is
\begin{equation}\label{f}
f(x,y)=F(x).
\end{equation}
On the other hand, any horizontal $y=c$ is mapped into a line through the origin $v=m_cu$ which can be rewritten as $g(x,c)=m_c\cdot f(x,c)$, from which $g(x,c)/f(x,c)=m_c$, where $m_c$ depends only on $c$. This means that $g(x,y)/f(x,y)=G(y)$, which after using (\ref f) becomes
\begin{equation}\label{g}
g(x,y)=F(x)\cdot G(y).
\end{equation}

Taking partial derivatives, we obtain
\begin{equation*}
f_x(x,y)=F'(x),\,\,g_x(x,y)=F'(x)\cdot G(y),\,\,f_y(x,y)=0,\,\,g_y(x,y)=F(x)\cdot G'(y).
\end{equation*}

Using equation \eqref {fxgx}, we have that
\begin{equation*}
f^2_x+g^2_x=f^2_y+g^2_y\rightarrow [F'(x)]^2+[F'(x)]^2\cdot G^2(y)=F^2(x)\cdot [G'(y)]^2\quad \text{from which}
\end{equation*}

\begin{equation*}
\left(\frac{F'(x)}{F(x)}\right)^2=\frac{G'(y)^2}{1+G^2(y)}.
\end{equation*}

Since the left hand side of the above equation depends on $x$ only while the right hand side depends on $y$ only,
it follow that we have system of differential equations
\begin{equation*}
\frac{F'(x)}{F(x)}=\pm k,\qquad \frac{G'(y)}{\sqrt{1+G^2(y)}}=\pm k.
\end{equation*}

One set of solutions of this system is given by
\begin{equation*}
F(x)=a^x,\qquad G(y)=\frac{a^y-a^{-y}}{2}.
\end{equation*}

Using now relations (\ref f) and (\ref g), we have obtained
\begin{equation}\label{Ttrapezoid}
T:\quad f(x,y)=a^x,\qquad g(x,y)=\frac{a^x(a^y-a^{-y})}{2}.
\end{equation}

We still have to verify property (c).
Consider the axes-parallel square $ABCD$ in the $xy$ plane with  $ A(x,y),\, B(x+l,y),\, C(x+l,y+l),\, D(x,y+l)$ where $l >0$. Denote $A' := T(A),\,\, B':= T(B),\,\, C' := T(C),\,\, D' := T(D)$. It is easy to verify that $A'B'C'D'$ is a positively oriented trapezoid. It remains to check that $A'B'C'D'$ is tangential. Denote $X=a^x,\,\, Y=a^y,\,\, L=a^l$. Obviously, $X>0, Y>0, L>1$. To this end, we compute the side lengths of $A'B'C'D'$.

\begin{eqnarray*}
A'B'&=&\frac{X(1+Y^2)(L-1)}{2Y},\,\,\,\,\,\, C'D'\,\,=\,\,\frac{X(1+Y^2L^2)(L-1)}{2YL}\\
B'C'&=&\frac{X(1+Y^2L)(L-1)}{2Y},\,\,D'A'\,\,=\,\,\frac{X(1+Y^2L)(L-1)}{2YL}.
\end{eqnarray*}

It is easy to verify that $A'B'+C'D' = B'C'+D'A'$, thus $A'B'C'D'$ is tangential. Finally, we show that given any tangential trapezoid $K$, an appropriate scaled and/or rotated copy of $K$ can be obtained as the image under transformation $T$ of a conveniently chosen square in the $xy$ plane. Without loss of generality, assume that the parallel sides of $K$ are vertical, and that the non-parallel sides intersect at the origin - see figure \ref{fig7}.
\begin{figure}[!htb]
\centering
\includegraphics[scale=.55]{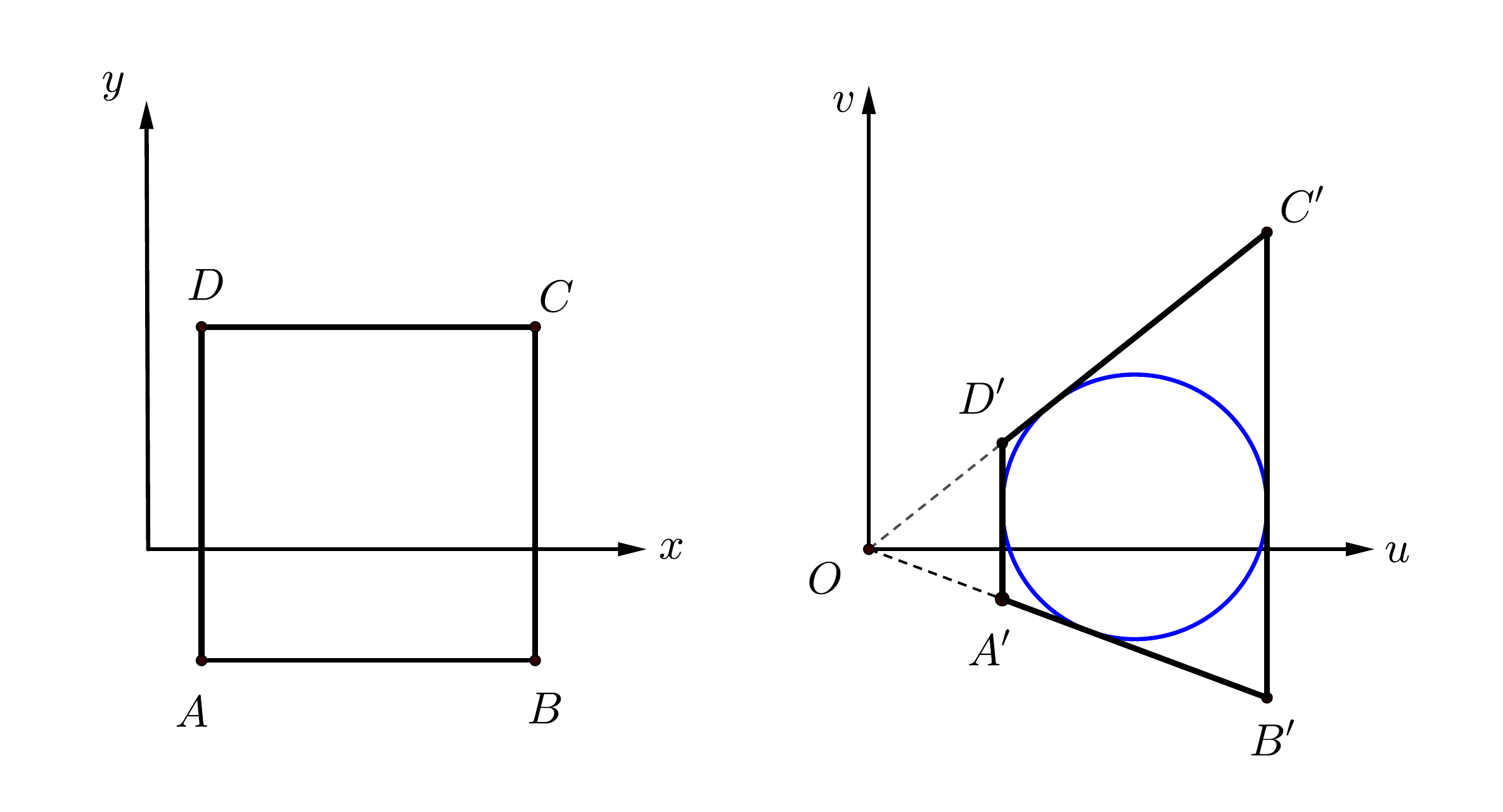}
\vspace{-0.2cm}
\caption{\small The square $ABCD$ being mapped into the tangential trapezoid $A'B'C'D'$. Vertical lines are mapped into vertical lines, horizontal lines are mapped into lines through the origin.}
\label{fig7}
\end{figure}

Let the slope of $A'B'$ be $m$, and let the slope of $C'D'$ be $p$, with $p>m$.  We have that the slope of $A'B' = \frac{Y^2-1}{Y}$, and the slope of $C'D' = \frac{Y^2L^2 - 1}{2YL}$. Choosing $X = 1$, $Y=m+\sqrt{1+m^2}$, and $L=(\sqrt{1+p^2}+p)(\sqrt{1+m^2}-m)$ achieves the desired goal.

\end{proof}
\end{section}

\begin{section}{\bf The general case}
A result similar to Theorem \ref{trapezoidthm} holds when the tangential quadrilateral to be dissected is not a trapezoid. We present the details below.

\begin{thm}\label{generalcase}
There exists a transformation, $T:\mathbb{R}^2 \rightarrow \mathbb{R}^2$ with the following properties

\noindent(a) $T$ maps vertical lines from the $xy$ plane into lines through the point $P(1,0)$ in the $uv$ plane.\\
(b) $T$ maps horizontal lines from the $xy$ plane into lines through the origin $O(0,0)$ in the $uv$ plane.\\
(c) $T$ maps any axes-parallel square lying in the halfplane $x+y>0$ of the $xy$ plane into a tangential quadrilateral in the $uv$ plane.
\end{thm}
\begin{proof}
Let $T:u=f(x,y),v=g(x,y)$.
Any vertical line $x=c$ in the $xy$ plane is mapped into a line $v=m_c(u-1)$, in the $uv$ plane.
This means $g(c,y)=m_c(f(c,y)-1)$, from which $g(c,y)/(f(c,y)-1)=m_c$.  It follows that $g(x,y)/(f(x,y)-1)$ depends on $x$ only, that is
\begin{equation}\label{ff}
g(x,y)=F(x)\cdot (f(x,y)-1).
\end{equation}
On the other hand, any horizontal $y=c$ is mapped into a line through the origin $v=m_cu$. This can be rewritten as $g(x,c)=m_c\cdot f(x,c)$, from which $g(x,c)/f(x,c)=m_c$, where $m_c$ depends only on $c$. This means that
\begin{equation}\label{gg}
g(x,y)/f(x,y)=G(y) \quad \text{from which} \quad g(x,y)=G(y)\cdot f(x,y).
\end{equation}

Solving (\ref{ff}) and (\ref{gg}), we obtain
\begin{equation} \label{ffgg}
f(x,y)=\frac{F(x)}{F(x)-G(y)},\qquad g(x,y)=\frac{F(x)\cdot G(y)}{F(x)-G(y)}.
\end{equation}

Taking partial derivatives in (\ref{ffgg}), we obtain
\begin{eqnarray*}
f_x(x,y)&=&\frac{-F'(x)\cdot G(y)}{[F(x)-G(y)]^2},\qquad g_x(x,y)=G(y)\cdot f_x(x,y),\\
f_y(x,y)&=&\frac{-F(x)\cdot G'(y)}{[F(x)-G(y)]^2},\qquad g_y(x,y)=F(x)\cdot f_y(x,y).
\end{eqnarray*}

Using the local condition \eqref{fxgx}, we have that
\begin{eqnarray*}
&f^2_x+g^2_x=f^2_y+g^2_y \longrightarrow (f_x(x,y))^2\cdot [1+G^2(y)]=(f_y(x,y))^2\cdot [1+F^2(x)]\rightarrow\\
& \rightarrow (F'(x))^2\cdot (G(y))^2\cdot [1+G^2(y)]=(F(x))^2\cdot (G'(y))^2\cdot [1+F^2(x)]\rightarrow\\
\end{eqnarray*}
\begin{equation*}
\longrightarrow \left (\frac{F'(x)}{F(x)\cdot \sqrt{1+F^2(x)}}\right)^2=\left(\frac{G'(y)}{G(y)\cdot \sqrt{1+G^2(y)}}\right)^2.
\end{equation*}

Since the left hand side of the above equation depends on $x$ only while the right hand side depends on $y$ only,
it follows that we have the following system of differential equations
\begin{equation*}
\frac{F'(x)}{F(x)\cdot \sqrt{1+F^2(x)}}=\pm k;\qquad \frac{G'(y)}{G(y)\cdot \sqrt{1+G^2(y)}}=\pm k.
\end{equation*}

One set of solutions of this system is given by
\begin{equation*}
F(x)=\frac{2a^x}{a^{2x}-1},\qquad G(y)=\frac{2a^{y}}{a^{2y}-1},\qquad {\rm where}\,\, a>1.
\end{equation*}

Substituting these into (\ref{ffgg}) we obtain the desired transformation

\begin{equation}\label{T}
T: \qquad f(x,y)=\frac{a^x(a^{2y}-1)}{(a^x+a^y)(a^{x+y}-1)},\qquad g(x,y)=\frac{2\,a^{x+y}}{(a^x+a^y)(a^{x+y}-1)}.
\end{equation}

\medskip

We still have to verify property (c).
Notice that transformation $T$ is not defined for points along the line $x+y=0$. It is only the squares which
are entirely contained in one of the two open half-planes determined by $x+y=0$ which have the desired property.

Consider the square $ABCD$ in the $xy$ plane with  $ A(x,y),\, B(x+l,y),\, C(x+l,y+l),\, D(x,y+l)$ where $l >0$. Further assume that $x+y>0$ so that the entire square $ABCD$ is above the line $x+y=0$.

Let $A' := T(A),\,\, B':= T(B),\,\, C' := T(C),\,\, D' := T(D)$ and denote $X:=a^x,\,\, Y:=a^y,\,\, L:=a^l$. Obviously, since $a>1$ and $x+y>0$ it immediately follows that $X>0, Y>0,\, XY>1, L>1$.
Thus, the coordinates of these four points can be written as
\begin{align*}
A'&= \left(\frac{X(Y^2-1)}{(X+Y)(XY-1)}, \frac{2XY}{(X+Y)(XY-1)}\right),\\
B'&= \left(\frac{LX(Y^2-1)}{(LX+Y)(LXY-1)}, \frac{2LXY}{(LX+Y)(LXY-1)}\right),\\
C'&= \left(\frac{LX(L^2Y^2-1)}{(LX+LY)(L^2XY-1)}, \frac{2L^2XY}{(LX+LY)(L^2XY-1)}\right),\\
D'&= \left(\frac{X(L^2Y^2-1)}{(X+LY)(LXY-1)}, \frac{2LXY}{(X+LY)(LXY-1)}\right).
\end{align*}
It is easy to verify that $A'B'C'D'$ is a positively oriented quadrilateral.

It remains to check that $A'B'C'D'$ is tangential. Calculate the side lengths of $A'B'C'D'$.
\begin{eqnarray*}
A'B'&=&\frac{XY(L-1)(Y^2+1)(X^2L+1)}{(X+Y)(XY-1)(XL+Y)(XYL-1)},\\
B'C'&=&\frac{XY(L-1)(Y^2L+1)(X^2L^2+1)}{(X+Y)(XL+Y)(XYL-1)(XYL^2-1)},\\
C'D'&=&\frac{XY(L-1)(Y^2L^2+1)(X^2L+1)}{(X+Y)(X+YL)(XYL-1)(XYL^2-1)},\\
D'A'&=&\frac{XY(1+X^2)(L-1)(Y^2L+1)}{(X+Y)(XY-1)(X+YL)(XYL-1).}
\end{eqnarray*}

It is easy to verify that $A'B'+C'D' = B'C'+D'A'$, thus by the Pitot-Steiner theorem it follows that $A'B'C'D'$ is tangential. This proves the theorem.
\end{proof}

As in the trapezoid case, we can show that given any tangential quadrilateral $A'B'C'D'$, an appropriate scaled and/or rotated copy of $A'B'C'D'$ in the $uv$ plane can be obtained as the image under transformation $T$ of a conveniently chosen square in the $xy$ plane. However, in this case the computations are more complicated. We present the details in the following theorem.

\begin{thm}\label{onto}
Let $A'B'C'D'$ be a tangential quadrilateral in the $uv$ plane. Without loss of generality assume that after an eventual scaling and/or rotation we have that
$A'B'\cap C'D' = O(0,\,0)$ and $B'C'\cap D'A' =P(1,\,0)$ as shown in figure \ref{fig8} below. Then for any $a>1$ there exists constants $x$, $y$ and $l$, with $x+y>0$, $l>0$ such that the image of the axes-parallel square of side $l$ with the bottom left corner located at $A(x,\,y)$ under the transformation $T$ defined by \eqref{T} is the quadrilateral $A'B'C'D'$.
\end{thm}
\begin{figure}[!htb]
\centering
\includegraphics[scale=.55]{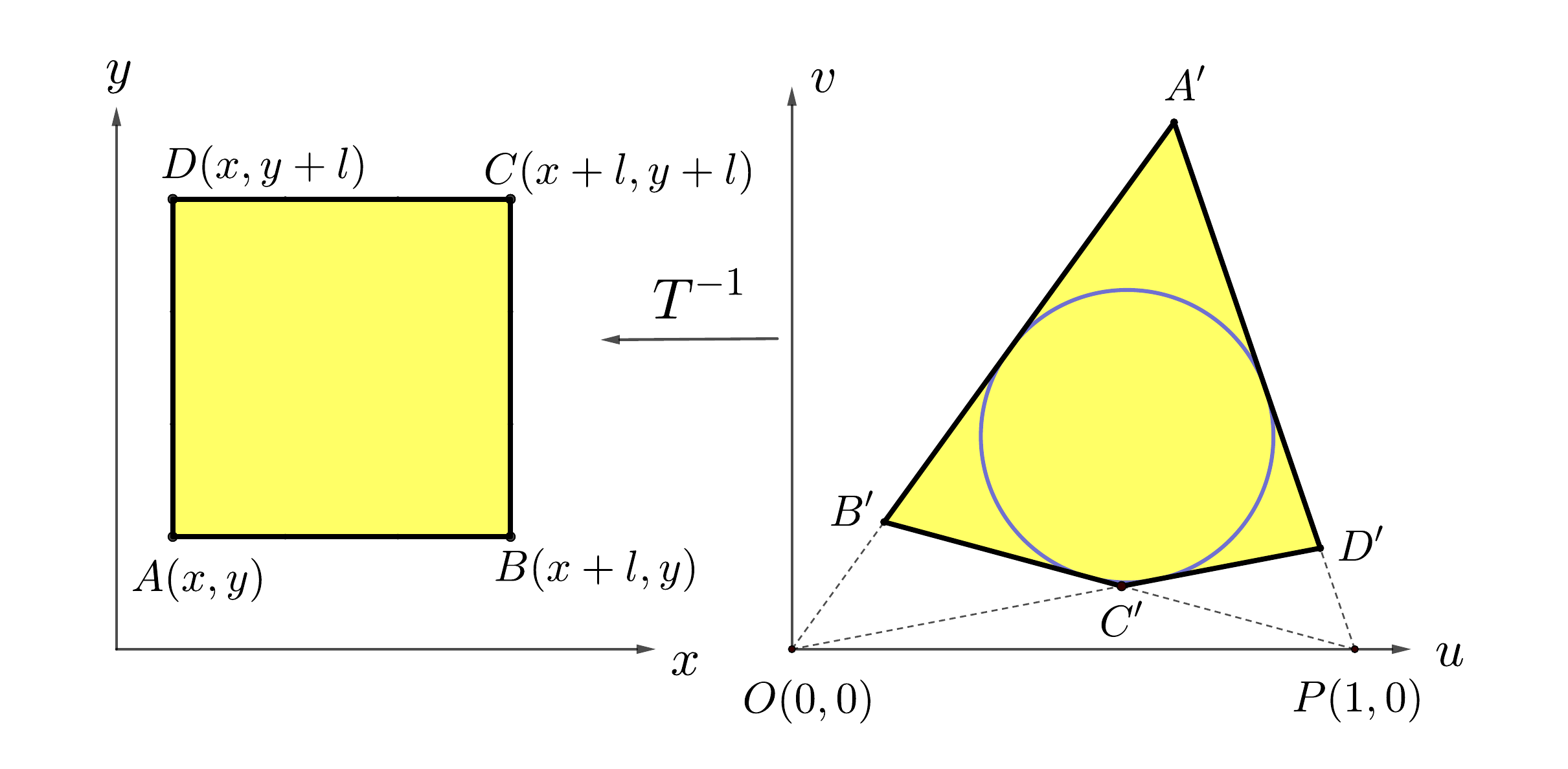}
\vspace{-0.5cm}
\caption{\small Proving that the transformation $T$ is surjective: every appropriately scaled/rotated tangential quadrilateral is the image of some axes-parallel square.}
\label{fig8}
\end{figure}
\begin{proof}
Denote by $t_1=\tan(\angle A'/2)$, $t_2=\tan (\angle B'/2)$, $t_3=\tan(\angle C'/2)$ and $t_4=\tan(\angle D'/2)$.
Obviously, $t_i>0$ for every $1\le i\le 4$. Since the angles of a quadrilateral sum up to $2\pi$ it follows easily
that the quantities $t_i$ satisfy the following identity
\begin{equation}
t_1+t_2+t_3+t_4= t_1\,t_2\,t_3+t_1\,t_2\,t_4+t_1\,t_3\,t_4+t_2\,t_3\,t_4.
\end{equation}

From our choice of $A'B'C'D'$ we have that $\angle A' +\angle B' <\pi$ and therefore $t_1\,t_2<1$. Similarly, since
$\angle A'+\angle D'<\pi$, it follows that $t_1\,t_4<1$.

We also need the following identity
\begin{equation}\label{identity}
t_1\,t_2+t_1\,t_4+t_2\,t_4-1= \frac{\cos(\angle C'/2)}{\cos(\angle A'/2)\,\cos(\angle B'/2)\,\cos(\angle D'/2)}.
\end{equation}
In particular, it follows that $t_1\,t_2+t_1\,t_4+t_2\,t_4-1>0$.
Hence,
\begin{equation}\label{positives}
1-t_1t_2>0,\,\, 1-t_1t_4>0,\,\, t_1\,t_2+t_1\,t_4+t_2\,t_4-1>0.
\end{equation}

Assume now that $ABCD$ is the axes-parallel square which is the desired pre-image of $A'B'C'D'$. Assume that this square has side $l>0$ with $A(x,\,y)$, $B(x+l,\,y)$, $C(x+l,\,y+l), D(x,\,y+l)$, and $x+y>0$.

We thus want to find $x, y$ and $l$ such that $A'=T(A)$, $B'=T(B)$, $C'=T(C)$, $D'=T(D)$. Denote $a^x=X$, $a^y=Y$ and $a^l=L$. Since $a>1$ the inequalities $x+y>0$ and
$l>0$ are equivalent to proving that $XY>1$ and $L>1$. Straightforward computations lead to the following equalities
\begin{align}\label{tangents}
\tan(\angle A'/2)&=t_1=\frac{X\,Y-1}{X+Y},\qquad\,\,\,\,\, \tan(\angle B'/2)=t_2=\frac{X\,L+Y}{XYL-1},\nonumber\\
\tan(\angle C'/2)&=t_3=\frac{XYL^2-1}{L(X+Y)},\qquad\tan(\angle D'/2)=t_4=\frac{X+YL}{XYL-1}.
\end{align}

After taking resultants of the expressions for $t_1$, $t_2$ and $t_4$ in \eqref{tangents} it follows that $X$, $Y$ and $L$ satisfy the following quadratic equations
\begin{align}
X^2 -\frac{2t_1+t_2-t_4-t_1^2(t_2+t_4)}{1-t_1\,t_2}\,X-1=0,\label{eqX}\\
Y^2 -\frac{2t_1-t_2+t_4-t_1^2(t_2+t_4)}{1-t_1\,t_4}\,Y-1=0,\label{eqY}\\
L^2 -\frac{(t_1\,t_2+t_1\,t_4-1)^2 +t_2^2+t_4^2+1}{t_1\,t_2+t_1\,t_4+t_2\,t_4-1}\,L +1=0.\label{eqL}
\end{align}
These equations are well defined as we earlier proved that $1-t_1\,t_2>0$, $1-t_1\,t_4>0$ and $t_1\,t_2+t_1\,t_4+t_2\,t_4-1>0$.
Let $X$, $Y$ and $L$ be the larger roots of \eqref {eqX}, \eqref {eqY} and \eqref {eqL}, respectively.

In particular,
\begin{align*}
X&=\frac{2t_1+t_2-t_4-t_1^2(t_2+t_4)+\sqrt{(1+t_1^2)\left[(t_1\,t_2+t_1\,t_4-2)^2+(t_2-t_4)^2\right]}}{2(1-t_1\,t_2)}\\
Y&=\frac{2t_1-t_2+t_4-t_1^2(t_2+t_4)+\sqrt{(1+t_1^2)\left[(t_1\,t_2+t_1\,t_4-2)^2+(t_2-t_4)^2\right]}}{2(1-t_1\,t_4)}\\
L&=\frac{(t_1\,t_2+t_1\,t_4-1)^2+t_2^2+t_4^2+1+(t_2+t_4)\sqrt{(1+t_1^2)\left[(t_1\,t_2+t_1\,t_4-2)^2+(t_2-t_4)^2\right]}}{2(t_1\,t_2+t_1\,t_4+t_2\,t_4-1)}.
\end{align*}
Ensuring that these choices for $X$, $Y$ and $L$ satisfy the system \eqref{tangents} is just a matter of algebraic calculation; for the value of $t_3$ one can use
identity \eqref{identity}. Such verifications can be easily performed using any computer algebra software (for instance we checked them using MAPLE).

It remains to prove that for these choices of $X, Y$ and $L$ we do have that $XY>1$ and $L>1$.

\noindent Let us introduce the following notations
\begin{align}
&a:=\frac{2t_1+t_2-t_4-t_1^2(t_2+t_4)}{1-t_1t_2},\label{a}\\
&b:=\frac{2t_1-t_2+t_4-t_1^2(t_2+t_4)}{1-t_1t_4},\label{b}\\
&c:=\frac{(t_1t_2+t_!t_4-1)^2+t_2^2+t_4^2+1}{t_1t_2+t_1t_4+t_2t_4-1}.\label{c}
\end{align}

The numbers $a$, $b$ and $c$ above are the coefficients of the linear terms in the quadratics \eqref{eqX}, \eqref{eqY} and \eqref{eqL}.
Hence, $X$, $Y$ and $L$ are the larger roots of the equations
\begin{equation}
X^2-aX-1=0,\quad Y^2-bY-1=0,\quad L^2-cL+1=0,\quad \text{that is,}
\end{equation}
\begin{equation}
X=\frac{a+\sqrt{a^2+4}}{2},\quad Y=\frac{b+\sqrt{b^2+4}}{2},\quad L=\frac{c+\sqrt{c^2-4}}{2}.
\end{equation}

Quick computations show that
\begin{equation*}\label{ab}
a+b=\frac{t_1\left((2-t_1t_2-t_1t_4)^2+(t_2-t_4)^2\right)}{(1-t_1t_2)(1-t_1t_4)},\,\, c=2+\frac{(2-t_1t_2-t_1t_4)^2+(t_2-t_4)^2}{(t_1t_2+t_1t_4+t_2t_4-1)}.
\end{equation*}

Combining the above equalities with \eqref{positives}, it follows that $a+b>0$ and $c>2$.

Then $a>-b$ and since the function $h(x)=x+\sqrt{x^2+4}$ is strictly increasing on $\mathbb{R}$ we have that $h(a)>h(-b)$ from which
\begin{equation*}
a+\sqrt{a^2+4}>-b+\sqrt{b^2+4},\quad \text{and after multiplying both sides by}\, b+\sqrt{b^2+4}
\end{equation*}
\begin{equation*}
(a+\sqrt{a^2+4})(b+\sqrt{b^2+4}) >4 \longrightarrow \frac{a+\sqrt{a^2+4}}{2}\cdot  \frac{b+\sqrt{b^2+4}}{2} >1\longrightarrow XY>1.
\end{equation*}
On the other hand, since $c>2$ it clearly follows that $L=(c+\sqrt{c^2-4})/2>1$, as claimed.

We thus proved that $XY>1$ and $L>1$ as needed. This completes the proof of the main result.
\end{proof}

It can be noticed that transformation $T$ given by \eqref {T} proves slightly more than that any given tangential quadrilateral
has an $n\times n$ class preserving grid dissection for every $n\ge 2$.
\begin{cor}
For any given tangential quadrilateral $Q$ and any dissection of an axes-parallel square into smaller squares, there exists a
topologically equivalent dissection of $Q$ into smaller tangential quadrilaterals.
\end{cor}
\end{section}

\begin{section}{\bf Three nice properties of tangential quadrilaterals}

There are three natural ``centers'' associated to every given tangential quadrilateral $A'B'C'D'$. First, we have $I$, the incenter of $A'B'C'D'$. Second, we have $S$,
the point of intersection of the diagonals $A'C'$ and $B'D'$. Third, there is $W$, the {\it $2\times\ 2$ center of $A'B'C'D'$}, the common point of the four smaller tangential
quadrilaterals which determine the $2\times 2$ grid dissection of $A'B'C'D'$ - see figure \ref{fig9}. One very interesting fact is that these points are always collinear.

\begin{thm}Let $A'B'C'D'$ be a tangential quadrilateral placed such that
$A'B'\cap C'D' = O(0,\,0)$ and $A'D'\cap B'C' =P(1,\,0)$ as shown in figure \ref{fig9}.

a) Let $I$, $S$ and $W$ be the three centers defined above. Then these points are collinear and the line determined by them is perpendicular to $OP$.

b) Denote the radii of the inscribed circles of the four small tangential quadrilaterals $A'B'C'D'$ is dissected into by $r_1$, $r_2$, $r_3$ and $r_4$, respectively.
Then
\begin{equation}\label{reciprocals}
\frac{1}{r_1}+\frac{1}{r_3}=\frac{1}{r_2}+\frac{1}{r_4}.
\end{equation}
\end{thm}

\begin{proof}
\vspace{-0.8cm}
\begin{figure}[!htb]
\centering
\includegraphics[scale=0.92]{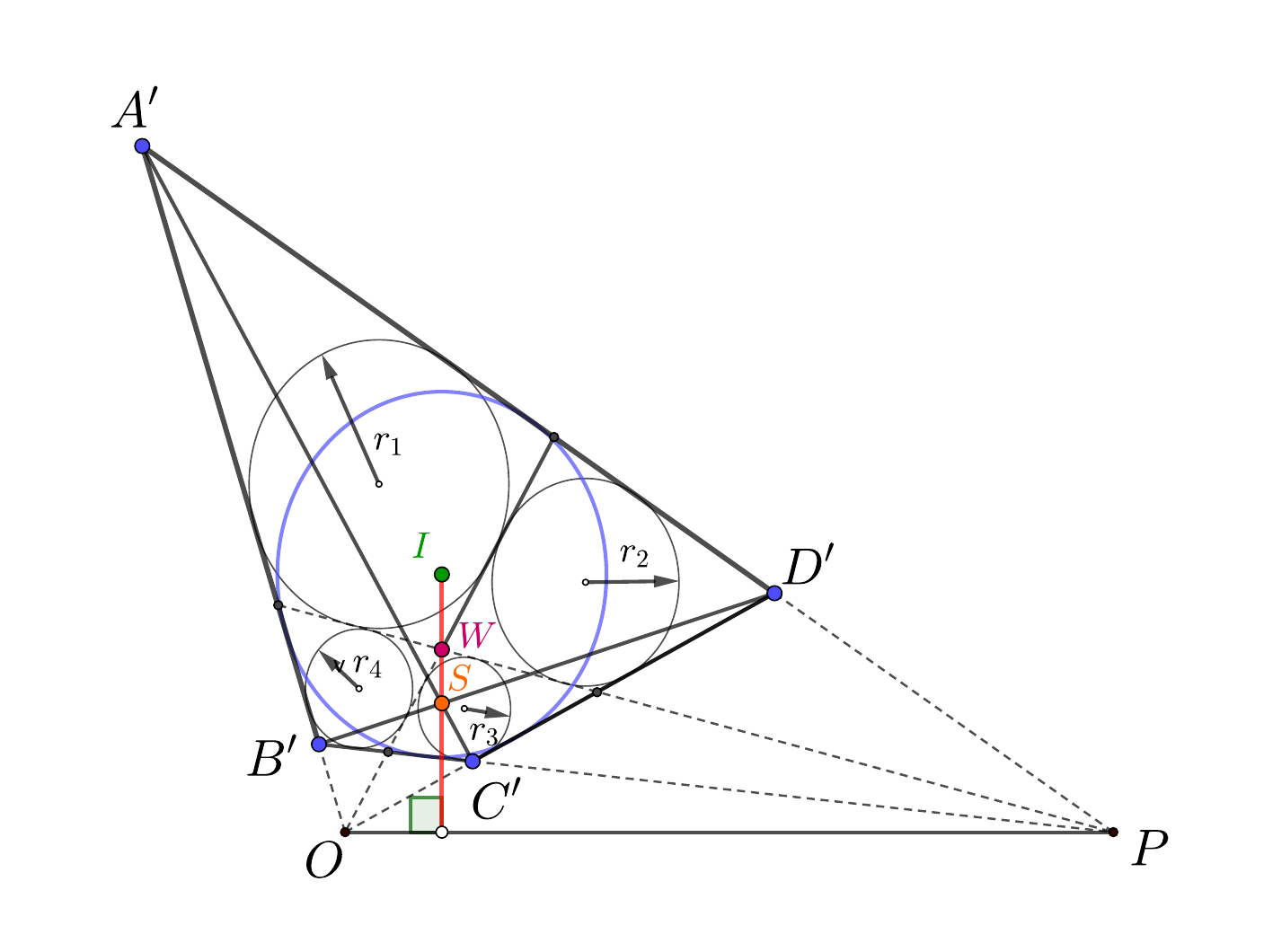}
\vspace{-0.7cm}
\caption{\small The incenter $I$, the $2\times 2$-center $W$,  and the point of intersection of the diagonals $S$, are collinear in any tangential quadrilateral}
\label{fig9}
\end{figure}

From theorems \ref {generalcase} and \ref {onto} it follows that there exists an axes-parallel square, $ABCD$ with  $ A(x,y),\, B(x+l,y),\, C(x+l,y+l),\, D(x,y+l)$, $a>1$, $l >0$, $x+y>0$ and a transformation $T$ given by \eqref{T} such that $T(ABCD)=A'B'C'D'$.
As  before, denote $X:=a^x,\,\, Y:=a^y,\,\, L:=a^{l}$. From our choices it follows that, $X>0, Y>0,\, XY>1$ and  $L>1$.

We just have to compute the coordinates of the points $I$, $S$ and $W$. Straightforward calculations show that all these points have the same abscissa
\begin{equation*}
u_I=u_S=u_W=\frac{X(L^2Y^2-1)}{(X+Y)(XYL^2-1)}, \quad \text{which proves the first assertion}.
\end{equation*}
For the second part, similar computations give the expressions of $r_i$, for $1\le i\le 4$.
\begin{align*}
r_1&= \frac{XY(L-1)}{(X+Y)(XYL-1)},\,\,\qquad r_2= \frac{XYL(L-1)}{(XL+Y)(XYL^2-1)},\\
r_3&= \frac{XYL(L-1)}{(X+YL)(XYL^2-1)}\qquad r_4= \frac{XYL^2(L-1)}{(X+Y)L(XYL^3-1)}.
\end{align*}

Equality \eqref {reciprocals} follows. It is interesting to mention that a similar relation is satisfied by the inradii of triangles $A'SB'$, $B'SC'$, $C'SD'$ and $D'SA'$.
This was proved in \cite{Wu} to be a necessary and sufficient condition for a quadrilateral to be tangential.
\begin{equation}
\frac{1}{r_{A'SB'}}+\frac{1}{r_{C'SD'}}=\frac{1}{r_{B'SC'}}+\frac{1}{r_{D'SA'}}.
\end{equation}
\end{proof}

Another property an $n\times n$ class preserving grid dissection has is what we call the
\emph{triple-grid property}. The situation is displayed below for the case $n=3$.
\begin{figure}[!htb]
\centering
\includegraphics[scale=1]{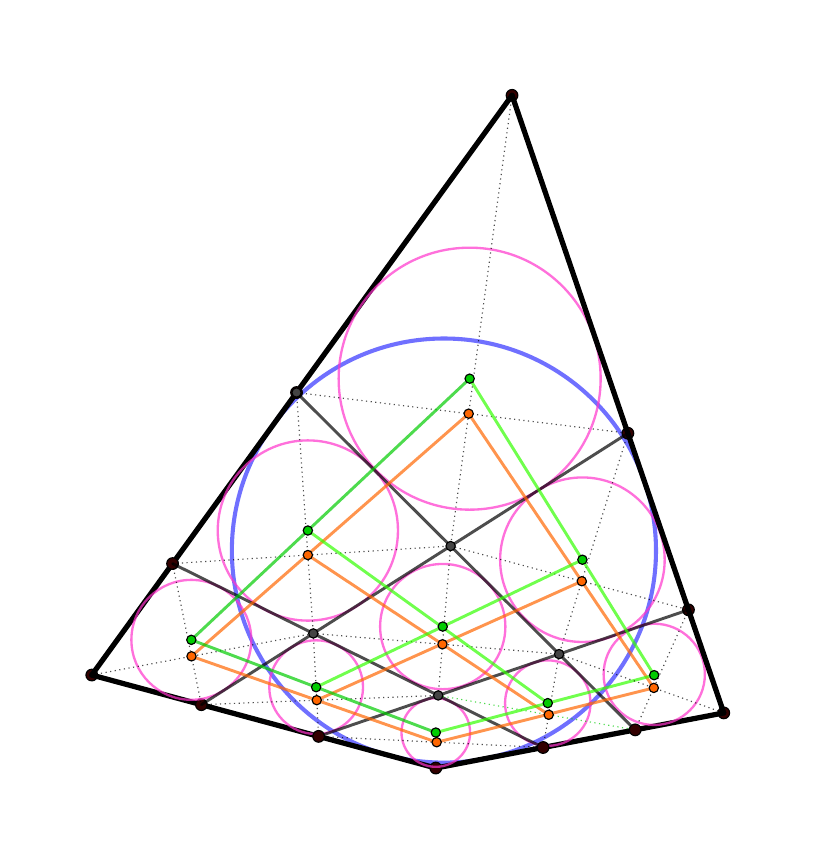}
\vspace{-1cm}
\caption{\small The triple grid property: both incenters and the points of intersection of the diagonals form an $n\times n$ grid.}
\label{fig10}
\end{figure}

The green points represent the incenters of the nine smaller quadrilaterals while the red points represent the intersections
of the diagonals within each of these nine quadrilaterals. It can be easily seen that the incenters as well as the
intersection points of the diagonals create a ``grid-like'' pattern. A third grid, not pictured in figure \ref{fig10}
in order to maintain clarity, is created by the $2\times 2$ centers of these nine quadrilaterals. The general proof is just a matter
of algebraic verification.
\end{section}

\end{document}